\documentclass[reqno]{amsart}

\usepackage{amsmath, amsfonts, amssymb, hyperref}
\usepackage{graphics}
\usepackage[usenames]{color}

{\theoremstyle{plain}
\newtheorem{theorem}{Theorem}[section]

\newtheorem{proposition}[theorem]{Proposition}
\newtheorem{lemma}[theorem]{Lemma}
\newtheorem{corollary}[theorem]{Corollary}
}
{\theoremstyle{definition}
\newtheorem{definition}[theorem]{Definition}
\newtheorem*{definition*}{Definition}
\newtheorem{remark}[theorem]{Remark}
\newtheorem{example}[theorem]{Example}

}

\numberwithin{equation}{section}

\newcommand{\bigP}[1]{\mathcal{P}_{#1}}
\newcommand{\multfreeP}[2]{\mathcal{M}_{#1,#2}}
\newcommand{\PNl}{\multfreeP{N}{\ell}}
\newcommand{\ribbons}[2]{\mathcal{R}_{#1,#2}}
\newcommand{\RNl}{\ribbons{N}{\ell}}

\newcommand{\leqs}{\leq_s}
\newcommand{\geqs}{\geq_s}
\newcommand{\les}{<_s}
\newcommand{\sh}{\mathit{sh}}
\newcommand{\rows}[1]{\mathrm{rows}(#1)}
\newcommand{\cols}[1]{\mathrm{cols}(#1)}
\newcommand{\leqdom}{\leq_{\mathit{dom}}}
\newcommand{\ledom}{<_{\mathit{dom}}}

\newcommand{\comment}[1]{\vspace{5 mm}\par \noindent
\marginpar{\textsc{Comment}}
\framebox{\begin{minipage}[c]{0.95 \textwidth}
 #1 \end{minipage}}\vspace{5 mm}\\}

\renewcommand{\comment}[1]{}

\newcommand{\new}[1]{\ensuremath{\blacktriangleright}#1\ensuremath{\blacktriangleleft}}

\renewcommand{\new}[1]{#1}

\begin{document}
\title[Positive ribbon differences]{Positivity results on ribbon Schur function differences}

\author{Peter R. W. McNamara}
\address{Department of Mathematics, Bucknell University, Lewisburg, PA 17837, USA}
\email{\href{mailto:peter.mcnamara@bucknell.edu}{peter.mcnamara@bucknell.edu}}

\author{Stephanie van Willigenburg}
\address{Department of Mathematics, University of British Columbia, Vancouver, BC V6T 1Z2, Canada}
\email{\href{mailto:steph@math.ubc.ca}{steph@math.ubc.ca}}
\thanks{The second author was supported in part by the National Sciences and Engineering Research Council of Canada.}

\begin{abstract} 
There is considerable current interest in
determining when the difference of two skew Schur functions is Schur positive.  
\new{We consider the posets that result from ordering skew diagrams according to Schur positivity, before focussing on the convex subposets corresponding to ribbons.}  While the general solution for ribbon Schur functions seems out of reach at present, we determine 
necessary and sufficient conditions for multiplicity-free ribbons, 
i.e.\ those whose expansion as a linear combination of Schur functions has all coefficients either 
zero or one.  
In particular, we show that the poset that results from ordering such ribbons according to Schur-positivity is essentially a product of two chains.
\end{abstract}

\maketitle

\comment{The abstract is now more involved.  I would be fine with changing it back.}

\section{Introduction}\label{sec:intro} 

For several reasons, the Schur functions can be said to be the most interesting and important
basis for the ring of symmetric functions.  While we will study Schur functions from a combinatorial
perspective, their importance is highlighted by their 
appearance in various other areas of mathematics.  They appear in the representation theory of the
symmetric group and of the general and special linear groups.  They are intimately tied to 
Schubert classes, which arise in algebraic geometry
when studying the cohomology ring of the Grassmannian, and they are also closely 
related to the eigenvalues of Hermitian matrices.  For more information on these and other
connections see, for example, \cite{Ful1} and \cite{Ful2}.

It is therefore natural to consider the expansions of other symmetric functions in the 
basis of Schur functions.  For example, the skew Schur function $s_{\lambda/\mu}$ and the product $s_\lambda s_\mu$ of two Schur functions are two of the most famous examples of
\emph{Schur positive} functions, i.e. when expanded as a linear combination of Schur functions, all
of the coefficients are non-negative.  Schur positive functions have a particular representation-theoretic
significance: if a homogeneous symmetric function of degree $N$ is Schur positive, then it 
arises as the Frobenius image of some representation of the symmetric group $S_N$. 
Motivated by the Schur positivity of $s_\lambda s_\mu$ and $s_{\lambda/\mu}$, one might ask when expressions of the form
\[
s_\lambda s_\mu - s_\sigma s_\tau \mbox{\ \ \ \ \ or\ \ \ \ \ } s_{\lambda/\mu} - s_{\sigma/\tau}
\]
are Schur positive, and such questions have been the subject of much recent work, such as
\cite{BBR, FFLP, KWvW, Kir, LPP, LLT, McN, Oko}.  It is well-known that these questions are
currently intractable when stated in anything close to full generality.  

Putting these questions in the following general setting will help put our work in context.  
Let us first note that $s_\lambda s_\mu$ is just a special type of skew Schur function \cite[p.~339]{ECII}.
Therefore, it suffices to consider differences of the form $s_A - s_B$, where $A$ and $B$ are skew
diagrams.  We could define a reflexive and transitive binary relation on skew Schur functions by
saying that $B$ is related to $A$ if $s_A - s_B$ is Schur positive.  To make this 
relation a partial order, we need to consider those skew diagrams that yield the same
skew Schur function to be equivalent; see the sequence \cite{HDL, HDL2, HDL3} for a study
of these equivalences.  Having done this, let us say that $[B] \leqs [A]$ if 
$s_A - s_B$ is Schur positive, where $[A]$ denotes the equivalence class of $A$.  Clearly $[A]$ and 
$[B]$ will be incomparable unless $A$ and $B$ have the same number $N$ of boxes, and we 
let $\bigP{N}$ denote the poset of all equivalence classes $[A]$ where the number of boxes in $A$ is $N$.  Restricting
to skew diagrams with 4 boxes, we get the poset $\bigP{4}$ shown in Figure~\ref{fig:p4}.
\begin{figure}[htbp]
\begin{center}
\[
\scalebox{.8}{\includegraphics{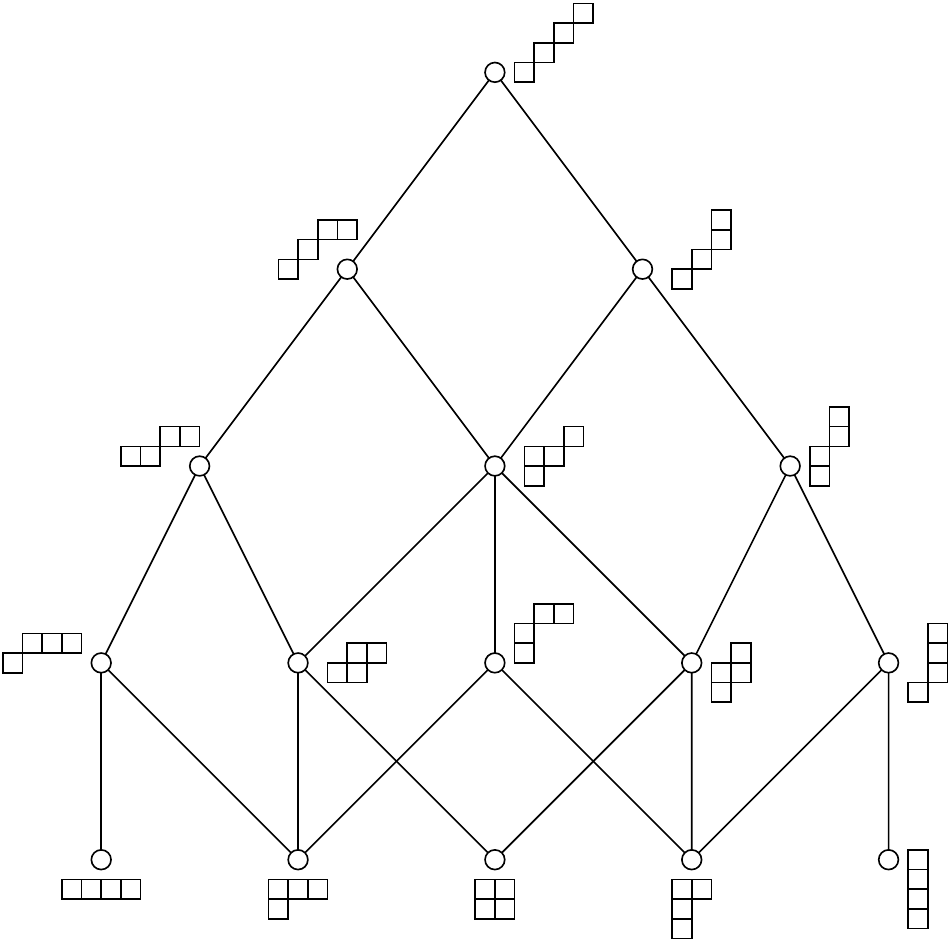}}
\]
\caption{$\bigP{4}$: All skew diagrams with 4 boxes under the Schur positivity order.  We note that $\bigP{N}$ is not graded for $N \geq 5$, and is not a join-semilattice for $N \geq 6$.}
\label{fig:p4}
\end{center}
\end{figure}
Our overarching goal when studying questions of Schur positivity and Schur equivalence
is to understand these posets.  

Our approach will be to restrict to a particular subposet of $\bigP{N}$ and derive necessary 
and sufficient conditions on $A$ and $B$ for $[B] \leqs [A]$.  This contrasts with most 
of the aforementioned papers, which studied either necessary \emph{or} sufficient conditions
for $[B] \leqs [A]$.  There are two previous examples in the literature of subposets of $\bigP{N}$
for which necessary and sufficient conditions are given.  The first example concerns 
the class of \emph{horizontal strips}
(respectively, vertical strips), which
consists of all skew diagrams with at most one box in each column (resp.\ row).  It is shown in
\cite[I.7 Example 9(b)]{Mac} that the poset that results when we restrict to horizontal (resp.\ vertical) 
 strips is exactly the dominance lattice on partitions of $N$.  The second 
 example concerns \emph{ribbons}, defined as connected skew diagrams with no 2-by-2 block of
 boxes.  In other words, ribbons  are connected skew diagrams with at most one box
 in each northwest to southeast diagonal.  
 As we will show in Lemma~\ref{lem:length}, two ribbons are incomparable in $\bigP{N}$ unless they
 have the same number of rows.  
 Restricting to ribbons whose row lengths 
 weakly decrease from top to bottom again results in dominance order, as shown in
 \cite[Theorem 3.3]{KWvW}.  More precisely, if $A$ and $B$ are ribbons with the same number of 
 rows and with row lengths weakly decreasing from top to bottom, then $[B] \leqs [A]$ if and
 only if the partition of row lengths of $A$ is less than or equal to the partition 
 of row lengths of $B$ in dominance order.
 
The next step would be to try to characterize the Schur positivity order for general ribbons.
More precisely, we would like conditions in terms of the diagrams of $A$ and $B$ that determine
whether or not $[B] \leqs [A]$.  \new{Understanding the ribbon case would give insight into many ``portions'' of $\bigP{N}$.  More precisely, we show that the subposet of $\bigP{N}$ consisting of ribbons with a fixed number of rows is a \emph{convex} subposet of $\bigP{N}$.  However, the general ribbon case is} extremely difficult: Figure~\ref{fig:ribbon94} shows 
that the set of ribbons with 9 boxes and 4 rows already yields a complicated poset.
\begin{figure}[htbp]
\begin{center}
\[
\scalebox{.5}{\includegraphics{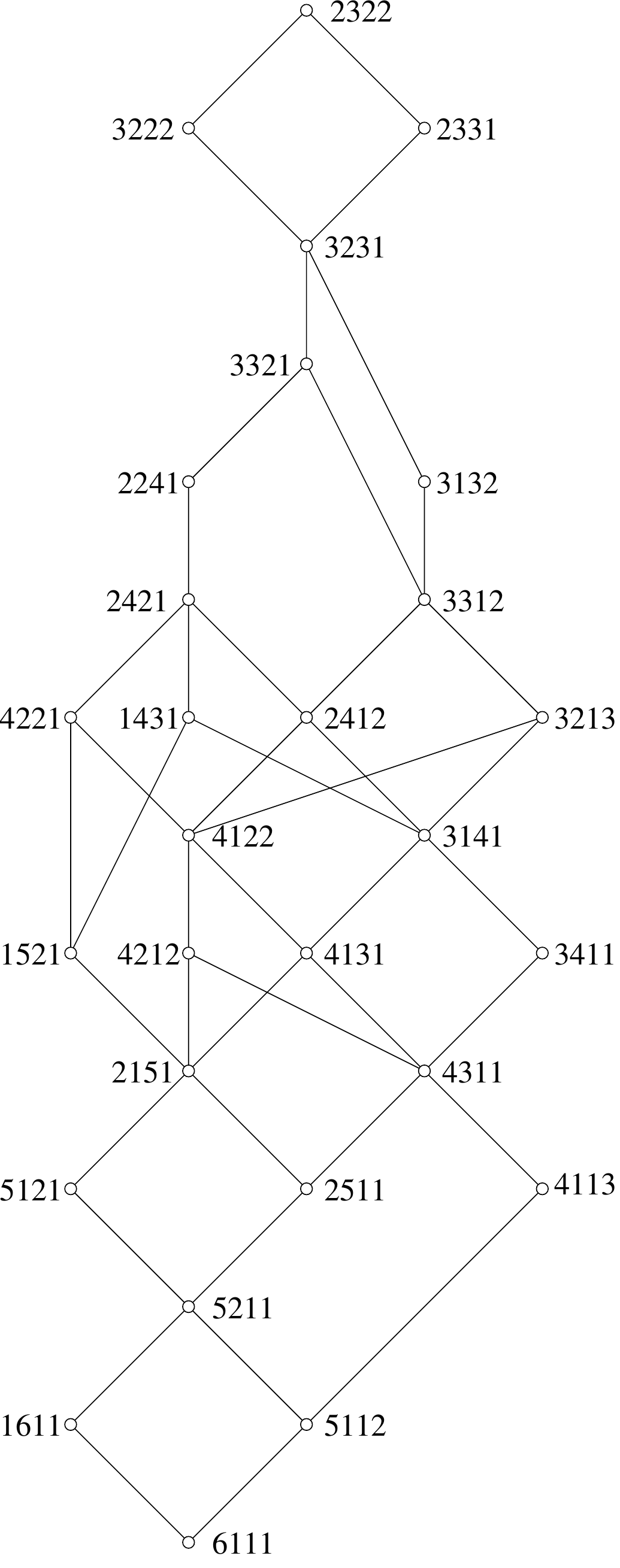}}
\]
\caption{$\ribbons{9}{4}$: all ribbons with 9 boxes and 4 rows under the Schur positivity order.
Ribbons are labelled by their sequence of row lengths, read from top to bottom.}
\label{fig:ribbon94}
\end{center}
\end{figure}

As progress towards a full characterization, we consider \emph{multiplicity-free} ribbons, i.e.\
ribbons whose corresponding skew Schur function, when expanded as a linear
combination of Schur functions, has all coefficients equal to 0 or 1.  \new{For details on the importance of  multiplicity-free linear combinations of Schur functions, see \cite{Ste} and the references given there.  For example, multiplicity-free Schur expansions correspond to multiplicity-free representations,
the many applications of which are studied in the survey article \cite{How}.}
Obviously, 
our first step is to determine which ribbons are multiplicity-free.  Conveniently, this
can be deduced from the work of Gutschwager and of Thomas and Yong; see 
Lemma~\ref{lem:multfreeribbon} for the details.  In short, a ribbon is multiplicity-free
if and only if it has at most
two rows of length greater than one and at most two columns of length greater than one.
Let $\PNl$ denote the poset that results from considering all multiplicity-free ribbons
with $N$ boxes and $\ell$ rows.  \new{As in the general ribbon case, we show that $\PNl$ appears as a convex subposet of $\bigP{N}$.}  Our main result is Theorem~\ref{th:bigdiff}, which gives
a complete description of the poset $\PNl$.  It turns out to have a particularly attractive form, 
and is only a slight modification of a product of two chains.  More precisely, $\PNl$ can be obtained from a product of two chains by removing some join-irreducible elements. As an example, Figure~\ref{fig:p12_6}
shows $\multfreeP{12}{6}$.
\begin{figure}[htbp]
\begin{center}
\[
\scalebox{.5}{\includegraphics{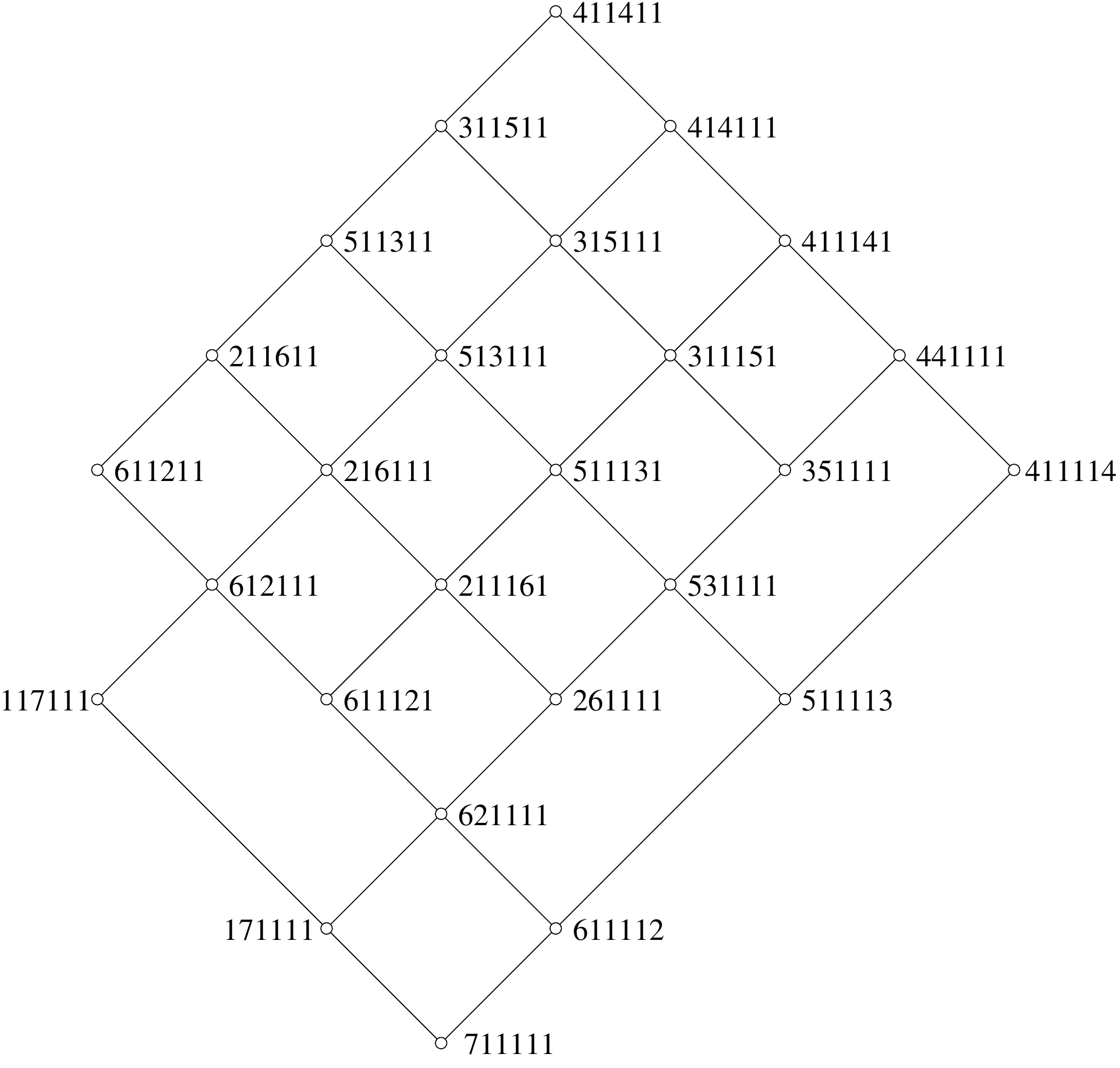}}
\]
\caption{$\multfreeP{12}{6}$: All multiplicity-free ribbons with 12 boxes and 6 rows under the Schur positivity order.  Ribbons are labelled by their sequence of row lengths, read from top to bottom.}
\label{fig:p12_6}
\end{center}
\end{figure}

The remainder of this paper is organized as follows.
In the next section, we introduce the necessarily preliminaries regarding partitions and skew Schur functions.  In Section~\ref{sec:reductions}, \new{we investigate general properties within $\bigP{N}$.  In particular,} we prove our earlier assertions about incomparability of ribbons, and we show that the sets of
ribbons and multiplicity-free ribbons each form convex subposets of $\bigP{N}$.  \new{We also show that
$\bigP{N}$, the ribbon subposets, and $\PNl$ each have a natural partition into convex subposets, where each convex subposet corresponds to a fixed multiset of row lengths.}
Section~\ref{sec:fundamental}
contains our main lemmas which determine the 
edges of the Hasse diagram of $\PNl$.  In Section~\ref{sec:main}, we reindex 
multiplicity-free ribbons in terms of certain rectangles, which allows us to state our main
result, Theorem~\ref{th:bigdiff}, fully describing all the order relations in $\PNl$.
This reindexing also explains why $\PNl$ closely resembles a product of two
chains, and it gives a simple description of the meet and join operations.  We conclude in 
Section~\ref{sec:conclusion} with some remarks about products of Schubert classes and 
a lattice-theoretic property of $\PNl$.  

\subsection{Acknowledgments} The authors would like to thank Hugh Thomas for helpful comments.
The Littlewood-Richardson calculator \cite{BucSoftware} and the posets package \cite{SteSoftware} were used for data generation.

\section{Preliminaries}\label{sec:prelim} 

\subsection{Partitions and diagrams}\label{subsec:partitions}

We begin by reviewing some notions concerning partitions. We say that a list of positive integers $\lambda =(\lambda _1,\lambda _2,\ldots ,\lambda _{\ell(\lambda)})$ is a \emph{partition} of a positive integer $N$ if $\lambda _1\geq \lambda _2 \geq \cdots \geq \lambda _{\ell(\lambda)}>0$ and $\sum _{i=1}^ {\ell(\lambda)} \lambda _i = N$. We denote this by $\lambda \vdash N$. We call $\ell(\lambda)$ the \emph{length} of $\lambda$ and we call $N$ the \emph{size} of $\lambda$, writing
$|\lambda| = N$.
Furthermore, we call the $\lambda _i$ the \emph{parts} of $\lambda$. If $\lambda _i = \lambda _{i+1}=\cdots = \lambda _{i+j-1}=a$ then we will denote the sublist $\lambda _i, \ldots, \lambda _{i+j-1}$ by $a^j$. For convenience we denote by $\emptyset$ the unique partition of length and size 0. Two partial orders that exist on partitions are
\begin{enumerate}
\item 
the \emph{inclusion order}: $\mu \subseteq \lambda$ if $\mu_i \leq \lambda_i$ for all $i=1,2,\ldots ,\ell(\mu)$,
\item
the \emph{dominance}  \emph{order}: Given $\lambda, \mu \vdash N$, $\mu \leqdom \lambda$ if 
$$
\mu_1 + \mu_2 + \cdots + \mu_i \leq \lambda_1 + \lambda_2 + \cdots + \lambda_i
$$
for $i=1,2,\ldots,\min\{\ell(\mu),\ell(\lambda)\}$.
\end{enumerate}
We say that a list of positive integers $\alpha =(\alpha _1,\alpha _2,\ldots ,\alpha _{\ell(\alpha)})$ is a \emph{composition} of $N$ if $\sum _{i=1}^{\ell(\alpha)} \alpha _i = N$. We denote this by $\alpha \vDash N$. As with partitions, a composition has length $\ell(\alpha)$ and size $|\alpha|$ with parts $\alpha _i$. We denote by $\alpha ^\ast$ the composition whose parts are the parts of $\alpha$ listed in reverse order, i.e. $\alpha ^\ast = (\alpha _{\ell(\alpha)},\ldots ,\alpha _2, \alpha _1)$.  Observe that every partition is a composition and that every composition \emph{determines a partition} $\lambda (\alpha)$, which is obtained by reordering the parts of $\alpha $ in weakly decreasing order.

Given a partition $\lambda =(\lambda _1,\lambda _2,\ldots ,\lambda _{\ell(\lambda)})$, we associate to it a \emph{diagram}, also denoted by $\lambda$, which consists of $\lambda _1$ left-justified boxes in the top row, $\lambda _2$ left-justified boxes in the second from top row etc. Given two partitions $\lambda,  \mu$ such that $\mu \subseteq \lambda$ we can associate to it a \emph{skew diagram} denoted by $\lambda/\mu$, which is obtained from the diagram $\lambda$ by removing the leftmost $\mu _i$ boxes from the $i$th row from the top, for $i=1,\ldots, \ell(\mu)$.

\begin{example}\label{ex:smallribbon}
The skew diagram for $A = \lambda/\mu = (4,3,3)/(2,2)$ is 
\setlength{\unitlength}{4mm}
\[
\begin{picture}(4,3)(0,0)
\put(2,3){\line(1,0){2}}
\put(2,2){\line(1,0){2}}
\put(0,1){\line(1,0){3}}
\put(0,0){\line(1,0){3}}
\multiput(2,3)(1,0){3}{\line(0,-1){1}}
\multiput(2,2)(1,0){2}{\line(0,-1){1}}
\multiput(0,1)(1,0){4}{\line(0,-1){1}}
\put(4.3,1.5){.}
\end{picture}
\]
\end{example}

There are two further partitions naturally associated with a skew diagram $A$.  We let $\rows{A}$
(resp.\ $\cols{A}$) denote the sequence of row (resp.\ column) lengths of $A$ ordered
into weakly decreasing order, and $\# \rows{A}$ (resp.\ $\#\cols{A}$) denote the number of rows (resp.\ columns) of non-zero length.  In Example~\ref{ex:smallribbon}, $\rows{A} = (3,2,1)$ and
$\cols{A} = (3,1,1,1)$. 
\setlength{\unitlength}{1.5mm}
We describe a skew-diagram as \emph{connected} if its boxes are edgewise connected, and we call it a \emph{ribbon} if it is connected and contains no subdiagram $(2,2)=
\begin{picture}(2,2)(0,0.4)
\multiput(0,0)(0,1){3}{\line(1,0){2}}
\multiput(0,0)(1,0){3}{\line(0,1){2}}
\end{picture}$\,.
Note that there exists a natural bijection $\psi$ between compositions of size $N$ and ribbons with $N$ boxes that takes the composition $\alpha =(\alpha _1,\alpha _2,\ldots ,\alpha _{\ell(\alpha)})$ and sends it to the unique ribbon that has $\alpha _i$ boxes in the $i$th row from the top. For example the skew diagram in Example~\ref{ex:smallribbon} is a ribbon and corresponds to the composition $(2,1,3)$. For ease of notation we will often refer to a ribbon by its corresponding composition from now on.

We conclude this subsection with two operations on skew diagrams. The first is antipodal rotation, Given a skew diagram $\lambda /\mu$ we  form its \emph{antipodal rotation} $(\lambda /\mu)^\ast$ by rotating $\lambda /\mu$ by 180 degrees in the plane. Observe that if $\lambda /\mu$ is a ribbon corresponding to a composition $\alpha$ then $\psi(\alpha ^\ast)=(\lambda/\mu)^\ast$. The second operation is transposition. Given a diagram $\lambda$ we  form its \emph{transpose} $\lambda^t$ by letting the leftmost column of $\lambda ^t$ have $\lambda _1$ boxes, the second from leftmost column have $\lambda _2$ boxes etc. We then extend this to skew diagrams by $(\lambda/\mu)^t:=\lambda ^t/\mu^t$.

\begin{example}\label{ex:transpose} If $\lambda/\mu=(4,3,3)/(2,2)$ then
\setlength{\unitlength}{4mm}
\vspace{-2.5\unitlength}
\[
(\lambda/\mu)^\ast = 
\begin{picture}(4,3)(0,1.5)
\put(1,3){\line(1,0){3}}
\put(1,2){\line(1,0){3}}
\put(0,1){\line(1,0){2}}
\put(0,0){\line(1,0){2}}
\multiput(1,3)(1,0){4}{\line(0,-1){1}}
\multiput(1,2)(1,0){2}{\line(0,-1){1}}
\multiput(0,1)(1,0){3}{\line(0,-1){1}}
\end{picture}
\quad\mbox{ and }\quad
(\lambda/\mu)^t =
\begin{picture}(3,4)(0,2)
\put(2,4){\line(1,0){1}}
\put(2,3){\line(1,0){1}}
\put(0,2){\line(1,0){3}}
\put(0,1){\line(1,0){3}}
\put(0,0){\line(1,0){1}}
\multiput(2,4)(1,0){2}{\line(0,-1){3}}
\multiput(0,2)(1,0){2}{\line(0,-1){2}}
\end{picture}\ .
\]
\vspace{\unitlength}
\end{example}

\subsection{Schur functions and skew Schur functions}\label{subsec:schur}

In this subsection we review necessary facts pertaining to the algebra of symmetric functions. We begin with tableaux. 

Consider a skew diagram $\lambda/\mu$. 
We say that we have a \emph{semi-standard Young tableau (SSYT)}, $T$, of \emph{shape} $\sh(T):=\lambda/\mu$ if the boxes of $\lambda /\mu$ are filled 
with positive integers such that
\begin{enumerate}
\item the entries of each row weakly increase when read from left to right,
\item the entries of each column strictly increase when read from top to bottom.
\end{enumerate}

\begin{example}\label{ex:ssyt}
The following is an SSYT of shape $\lambda/\mu = (4,3,3)/(2,2)$:
\setlength{\unitlength}{4mm}
\[
\begin{picture}(4,3)(0,0)
\put(2,3){\line(1,0){2}}
\put(2,2){\line(1,0){2}}
\put(0,1){\line(1,0){3}}
\put(0,0){\line(1,0){3}}
\multiput(2,3)(1,0){3}{\line(0,-1){1}}
\multiput(2,2)(1,0){2}{\line(0,-1){1}}
\multiput(0,1)(1,0){4}{\line(0,-1){1}}
\put(2.3,2.25){1}
\put(3.3,2.25){1}
\put(2.3,1.25){2}
\put(2.3,0.25){3}
\put(1.3,0.25){2}
\put(0.3,0.25){1}
\put(4.3,1.5){.}
\end{picture}
\]
\end{example}

Given an SSYT, $T$, we define its \emph{reading word}, $w(T)$, to be the entries of $T$ read from right to left and top to bottom. If, for all positive integers $i$ and $j$, the first $j$ letters of $w(T)$ includes at least as many $i$'s as $(i+1)$'s, 
then we say that $w(T)$ is \emph{lattice}. If we let $c_i(T)$ be the total number of $i$'s appearing in $T$, and so also in $w(T)$, then the list $c(T):=(c_1(T),c_2(T),\ldots)$ is called the \emph{content} of $T$ and also of $w(T)$.  For the SSYT $T$ in Example~\ref{ex:ssyt}, $c(T)=(3,2,1)$ and $w(T)=(1,1,2,3,2,1)$, and one can check that $w(T)$ is lattice.

With this in mind we can now define Schur functions and skew Schur functions.
For $\lambda \vdash N$, the Schur function $s_\lambda$ in the variables $x_1,x_2,\ldots$
is defined by
$$s_\lambda :=\sum _T x^T$$ where the sum is over all SSYT $T$ with $\sh(T)=\lambda$, and $x^T:=x_1^{c_1(T)}x_2^{c_2(T)}\cdots$.  We also let $s_{\emptyset}=1$.
It can be shown that $s_\lambda$ is symmetric in its variables $x_1, x_2,\ldots$.  Furthermore, working over $\mathbb{Q}$ say, the set $\{s_\lambda\}_{\lambda \vdash N}$ spans the space $\Lambda^N$ consisting
of all homogeneous symmetric functions of degree $N$ in the variables $x_1, x_2, \ldots$.
We also define the algebra of symmetric functions by $\Lambda := \bigoplus _{N\geq 0} \Lambda ^N$.
By extending our indexing set from diagrams to skew diagrams we create \emph{skew Schur functions}
$$s_{\lambda/\mu} :=\sum _T x^T\in \Lambda$$ where the sum is over all SSYT $T$ with $\sh(T)=\lambda/\mu$. If $\lambda/\mu$ is a ribbon then we call $s_{\lambda /\mu}$ a \emph{ribbon Schur function} and denote it by $r(\alpha)$, where $\alpha$ is the composition satisfying $\psi(\alpha)=\lambda/\mu$. 

It is a well-known fact that the $\{ s_\lambda \} _{\lambda \vdash N}$ not only span $\Lambda^N$ but are, in fact, a basis. Hence a natural question to ask is how does a skew Schur functions expand in terms of Schur functions. The answer is provided by the \emph{Littlewood-Richardson rule} which states
\begin{equation}\label{eq:lrrule}
s_{\lambda /\mu}=\sum _\nu c^\lambda _{\mu\nu} s_\nu
\end{equation}where $c^\lambda _{\mu\nu}$ is the number of SSYT with $\sh(T)=\lambda /\mu$ such that
\begin{enumerate}
\item $c(T)=\nu$,
\item $w(T)$ is lattice.
\end{enumerate}
For this reason, we will call an SSYT $T$ such that $w(T)$ is lattice a \emph{Littlewood-Richardson filling}, or \emph{LR-filling} for short.
If $c^\lambda _{\mu\nu}$ is $0$ or $1$ for all $\nu$ then we say $s_{\lambda /\mu}$ is \emph{multiplicity-free} and also that $\lambda/\mu$ is \emph{multiplicity-free}.

\begin{example}\label{ex:lrrule}
$$s_{(3,2,1)/(2,1)}=s_{(3)}+2s_{(2,1)}+s_{(1,1,1)}$$and
$$s_{(2,2)/(1)}=s_{(2,1)}.$$
Observe that the first example is not multiplicity-free, whereas the second example is. The second example can also be described as the ribbon Schur function $r(1,2)$.  
\end{example}

It is clear that \eqref{eq:lrrule} is a non-negative linear combination of Schur functions, which motivates our last definition.

\begin{definition}\label{def:schurpos}
If a symmetric function $f\in \Lambda$ can be written as a non-negative linear combination of Schur functions then we say that $f$ is \emph{Schur positive}. If $f$ can be written as a non-positive linear combination of Schur functions then we say that $f$ is \emph{Schur negative}. If $f$ is neither Schur positive or Schur negative then we say that $f$ is \emph{Schur incomparable}.
\end{definition} 

As an example, we know from Example~\ref{ex:lrrule} that $$s_{(3,2,1)/(2,1)}-s_{(2,2)/(1)}$$ is Schur positive.

The antipodal rotation and transpose operation that concluded the previous subsection can also
be interpreted in terms of skew Schur functions.  We first observe that skew Schur functions
are preserved under antipodal rotation.

\begin{proposition}\cite[Execrcise 7.56(a)]{ECII}\label{prop:ribreverse}
If $A$ is a skew diagram, then $s_A = s_{A^\ast}$.
\end{proposition}

Turning to the transpose operation, we 
recall the involution $\omega: \Lambda \rightarrow \Lambda$ defined on Schur functions by
$\omega(s_\lambda)=s_{\lambda ^t}$.  It extends to skew Schur functions to give
$$\omega(s_{\lambda/\mu})=s_{(\lambda/\mu)^t}.$$

\section{Subposets of ribbons}\label{sec:reductions}

Our goal for this section is to determine some general facts about the set of
multiplicity-free ribbons and its structure within $\bigP{N}$.  We will
consider results in decreasing order of generality: those that hold for
skew diagrams, then those that hold for ribbons, and finally those that apply to 
multiplicity-free ribbons.  In this spirit, we begin with some necessary conditions on skew diagrams $A$ and $B$ for $s_A - s_B$ to be Schur positive. 

\begin{lemma}\label{lem:size}
Let $A$ and $B$ be skew diagrams. If $|A|\neq |B|$ then $s_A - s_B$ is Schur incomparable.
\end{lemma}

\begin{proof}
This follows immediately from \eqref{eq:lrrule} which implies that any $s_\lambda$ appearing in the Schur function expansion of $s_A$ satisfies $|\lambda |= |A|$.
\end{proof}

The next lemma will justify several upcoming deductions.

\begin{lemma}\label{lem:rowscols}
Let $A$ and $B$ be skew diagrams.  If $s_A - s_B$ is Schur positive, then 
\[
\rows{A} \leqdom \rows{B} \mbox{\ \ \ and\ \ \ } \cols{A} \leqdom \cols{B}.
\]
Furthermore, for any fixed $m$ and  $n$, 
the number of $m$-by-$n$ rectangular subdiagrams contained inside $A$ is less
than or equal to the number for $B$.  
\end{lemma}

\begin{proof}
The latter assertion is one of the main results of \cite{McN}.
Since the first pair of inequalities are well-known folklore results that are difficult to find in the literature, they have recently
been reproduced with proof in \cite{McN}.
\end{proof}

As promised, it is now time to restrict our attention to ribbons.

\begin{lemma}\label{lem:length}
Let $\alpha, \beta$ be compositions. If $\ell(\alpha) \neq \ell(\beta)$ then $r(\alpha)-r(\beta)$ is Schur incomparable.
\end{lemma}

\begin{proof}
By Lemma~\ref{lem:size}, we can assume that $|\alpha| = |\beta|$.
Suppose, without loss of generality, that $\ell(\alpha) < \ell(\beta)$.  Then we know that 
$\rows{\alpha} \not\leqdom \rows{\beta}$ and thus, by Lemma~\ref{lem:rowscols}, $r(\alpha) - r(\beta)$
is not Schur positive.  

Observe that for a ribbon $\alpha$, $\#\cols{\alpha} + \ell(\alpha) = |\alpha|+1$.  Therefore, 
$\beta$ has fewer columns than $\alpha$.  In particular, 
$\cols{\beta} \not\leqdom \cols{\alpha}$ and thus, again by Lemma~\ref{lem:rowscols}, 
$r(\beta) - r(\alpha)$ is not Schur positive.
\end{proof}

From Lemmas~\ref{lem:size} and \ref{lem:length} it follows that we need only consider differences of the form $r(\alpha)-r(\beta)$ where $|\alpha |=|\beta|=N$ and $\ell(\alpha) = \ell(\beta)=\ell$, as any other difference of ribbon Schur functions will be Schur incomparable. From this point on, we will consider ribbons with the same skew Schur function to be equivalent.  Rather than continually referring to equivalence classes of ribbons, we will simply refer to ribbons with the implicit understanding that 
a ribbon represents its equivalence class. 

\begin{definition}
Let $\RNl$ denote the poset 
whose elements are  $$\{\alpha\ |\ \alpha\vDash N, \ell(\alpha)=\ell\}$$ 
subject to the relation 
$\alpha \geqs \beta$ if and only if $r(\alpha)- r(\beta)$ is Schur positive.
\end{definition}

We next observe that the elements of $\RNl$ all occur ``together'' in $\bigP{N}$.  More precisely, a subposet $Q$ of a poset $P$ 
is said to be \emph{convex} if, for all $a < b < c$ in $P$ with $a, c \in Q$, we have $b \in Q$.
We then have the following result.

\begin{proposition}\label{prop:ribbonsconvex}
$\RNl$ is a convex subposet of $\bigP{N}$.
\end{proposition}

\begin{proof}
Suppose $\alpha, \gamma \in \bigP{N}$ are also elements of $\RNl$, and $B$ is a 
skew diagram satisfying $\alpha \les B \les \gamma$.
Since $\alpha \les B$, the last part of Lemma~\ref{lem:rowscols} with $m=n=2$ tells us
that $B$ contains no 2-by-2 rectangular subdiagram.  Thus if $B$ is connected, we conclude that $B$ 
must be a ribbon.  Therefore, we can suppose $B$ is not connected.
By Lemma~\ref{lem:rowscols}, we have
\[
\rows{\gamma} \leqdom \rows{B}. 
\]
In particular, $B$ must have at most $\ell$ non-empty rows.  We also have that
\[
\cols{\gamma} \leqdom \cols{B}.
\]
Therefore, $B$ must have at most $\#\cols{\gamma}$
non-empty columns.   

Putting this together, 
we deduce that 
\[\#\cols{B} + \#\rows{B} \leq \#\cols{\gamma}+\ell = |\gamma|+1 = |B|+1,\] 
where we do not count empty columns and
rows of $B$.  On the other hand, since $B$ is not connected and has no 2-by-2 subdiagram, we see that we
must have $\#\cols{B} + \#\rows{B} > |B|+1$, a contradiction.
\end{proof}

It is now time to focus our attention on multiplicity-free ribbons.

\begin{definition}
Let $\PNl$ denote the poset whose elements are 
the multiplicity-free ribbons with $N$ boxes and $\ell$ rows,
subject to the relation 
$\alpha \geqs \beta$ if and only if $r(\alpha)- r(\beta)$ is Schur positive.
\end{definition}

In other words, $\PNl$ is the multiplicity-free part of $\RNl$.  
The reader may wish to find the 
various $\multfreeP{4}{\ell}$ in $\bigP{4}$ by referring to Figure~\ref{fig:p4}, and see Figure~\ref{fig:p12_6} for a more substantial example of $\PNl$.
By Lemmas~\ref{lem:size} and ~\ref{lem:length}, $\PNl$ and $\multfreeP{N'}{\ell'}$ are completely
incomparable unless $N=N'$ and $\ell=\ell'$.  Hence, from now on we fix $N$ and $\ell$ and restrict
our attention to $\PNl$. 

\begin{remark}\label{rem:transposeposet}
Observe by \eqref{eq:lrrule} and the definition of $\omega$ that we have $s_{\lambda/\mu} - s_{\sigma/\tau}$ is Schur positive if and only if $\omega(s_{\lambda/\mu}) - \omega(s_{\sigma/\tau})$ is Schur positive. Hence  applying $\omega$ to each element of $\PNl$ yields the poset $\multfreeP{N}{N-\ell+1}$.
\end{remark}

\begin{corollary}\label{cor:pnlconvex}
$\PNl$ is a convex subposet of $\bigP{N}$.
\end{corollary}

This corollary adds weight to our study of $\PNl$.  Once we show that $\PNl$ has a certain structure, 
the corollary tells us that this structure will not be ``hidden'' in $\bigP{N}$.  On the contrary, there
will be a copy of $\PNl$ appearing as a convex subposet of $\bigP{N}$ for every $\ell = 1, \ldots, N$.  

\begin{proof}[Proof of Corollary~\ref{cor:pnlconvex}]
By Proposition~\ref{prop:ribbonsconvex}, $\RNl$ is a convex subposet of $\bigP{N}$.  As a subposet
of $\RNl$, $\PNl$ must form an order ideal (or ``down-set'') since any element that is
less than a multiplicity-free element must itself be multiplicity-free.  Therefore, $\PNl$ is a
convex subposet of a convex subposet, and thus is a convex subposet of $\bigP{N}$.  
\end{proof}

Now that we have reduced the number of differences we need to consider by restricting to $\PNl$, our next step is to identify the ribbons that index these multiplicity-free ribbon Schur functions.

\begin{lemma}\label{lem:multfreeribbon}
If $\alpha \vdash 0$ then $r(\alpha)$ is multiplicity-free. If $\alpha \vDash N\geq1$ then $r(\alpha)$ is multiplicity-free if and only if $\alpha=(m, 1^k,n,1^l)$ or $\alpha=(m, 1^k,n,1^l)^\ast$ for $n
\geq 1$ and $k,l, m\geq 0$.
\end{lemma} 

\begin{proof}
The first part follows since $r(\emptyset)=1$. The second part follows from  \cite[Theorem 3.5]{Gutschwager} or \cite[Theorem 1]{ThomasYong}.
\end{proof}

The differences we need to consider are further reduced due to Proposition~\ref{prop:ribreverse}.
Therefore, it suffices to restrict to ribbons of the form $\alpha=(m, 1^k,n,1^l)$, with $n \geq 1$
and $k, l, m \geq 0$.  In fact, for our purposes, it is safe to ignore the case when $m=0$.  Indeed, this restriction only eliminates the ribbon $\alpha=(N)$, 
which is the unique ribbon in $\multfreeP{N}{1}$ and so is incomparable to all other ribbons.  Observe that ribbons of the form $(m, 1^k,n,1^l)$, when rotated 45 degrees clockwise, are typically in 
the shape of the letter M.  This is one of the reasons for our notation $\PNl$.
It is natural to consider whether we can reduce the number of differences to consider any further by discovering other equalities between multiplicity-free ribbon Schur functions. However, no others exist by \cite[Theorem 4.1]{HDL}.  In other words, within $\PNl$, the only members of the equivalence
class of a ribbon $\alpha$ are $\alpha$ and $\alpha^\ast$.

Before moving on to study individual order relations in $\PNl$, there is one more observation 
worth making about the structure of $\bigP{N}$.
The next result shows that the posets $\bigP{N}$, $\RNl$ and $\PNl$ themselves break up into convex subposets, with each such convex subposet corresponding to a fixed partition of row lengths.

\begin{proposition}
Given a partition $\lambda$ of $N$ with $\ell(\lambda)=\ell$, the set
\[
\{A \in \bigP{N}\ |\ \rows{A} = \lambda\}
\]
forms a convex subposet of $\bigP{N}$.  Furthermore, the intersection of this set with $\PNl$ 
(resp.\ $\RNl$) forms a convex subposet of $\PNl$ (resp.\ $\RNl$).
\end{proposition}

The reader may wish to observe this phenomenon in Figures~\ref{fig:p4}--\ref{fig:p12_6}.
The proposition also holds with $\cols{A}$  in place of $\rows{A}$.

\begin{proof}
Let $Q_\lambda$ denote the set  $\{A \in \bigP{N} \ |\ \rows{A} = \lambda\}$.
If $A \les B \les C$ and $A, C \in Q_\lambda$, then by Lemma~\ref{lem:rowscols}, 
\[
\rows{C} \leqdom \rows{B} \leqdom \rows{A}.
\]
This implies that $\rows{B} = \lambda$ and so $Q_\lambda$ is a convex subposet of $\bigP{N}$.  Furthermore, applying Corollary~\ref{cor:pnlconvex}, the intersection of $Q_\lambda$ with $\PNl$ must 
be convex in $\bigP{N}$ since the intersection of convex subposets is convex.  The convexity of the intersection in $\bigP{N}$ automatically implies its convexity in $\PNl$.  A similar argument 
that uses Proposition~\ref{prop:ribbonsconvex} applies to $\RNl$.  
\end{proof}

In conclusion, not only does each $\PNl$ sit nicely as a convex subposet of $\bigP{N}$,
each $\PNl$ consists entirely of convex subposets of the form 
\[
\{\alpha\ |\  \alpha \mbox{ is multiplicity-free}, \rows{\alpha} = \lambda\}.
\]

\section{Fundamental Relations}\label{sec:fundamental}

As a consequence of the previous section, 
we can focus our attention on differences of the form
$$r(\alpha)-r(\beta)$$
where $|\alpha |=|\beta|$, $\ell(\alpha) = \ell(\beta)$, $\alpha=(m, 1^k,n,1^l)$, $\beta=(m', 1^{k'} ,n',1^{l'})$, $m,n,m',n' \geq 1$ and $k,l,k',l' \geq 0$.  We are now ready to state four pivotal Schur positive differences.  The reader may wish to compare the left-hand side of these 
relations with the edges of Figure~\ref{fig:p12_6}.  The first (resp.\ last) two equalities 
correspond to the edges that run northeast (resp.\ northwest).

\begin{lemma}\label{lem:fourcovers}
\

\begin{enumerate}
\item If $n-1>m$ then 
$$r(n-1, 1^k, m+1, 1^l)-r(m, 1^k, n, 1^l)=\sum _{i=0} ^{\min \{k,l\}}s_{(n-1,m+1,2^i,1^{k+l-2i})}.$$
\item If $n>m$ and $l\geq 1$ then 
$$r(m, 1^k, n, 1^l)-r(n, 1^k, m, 1^l)=\sum _{i=0} ^{\min \{k,l-1\}}s_{(n,m+1,2^i,1^{k+l-2i-1})}.$$
\item If $n\geq 2$ and $l>k$ then 
$$r(m, 1^k, n, 1^l)-r(m, 1^l, n, 1^k)=\sum _{i=0} ^{\min \{n-2,m-1\}}s_{(n+m-i-1,i+2,2^k,1^{l-k-1})}.$$
\item If $m,n\geq 2$ and $l-1>k$ then 
$$r(m, 1^{l-1}, n, 1^{k+1})-r(m, 1^k, n, 1^l)=\sum _{i=0} ^{\min \{n-2,m-2\}}s_{(n+m-i-2,i+2,2^{k+1},1^{l-k-2})}.$$
\end{enumerate}
\end{lemma}

\begin{proof}
We begin by proving the first part. Let $\mathcal{T}$ be the set of all tableaux contributing towards the positive coefficient of some Schur function in the Schur function expansion of $r(m,1^k,n,1^l)$. 
Then $\mathcal{T}$ is the set of SSYT $T$ with $\sh(T)=(m, 1^k, n, 1^l)$ and $w(T)$ is lattice. 
Note that every $T\in \mathcal{T}$ has the form
$$\begin{matrix}
&&&&&&1&\cdots&1\\
&&&&&&2\\
&&&&&&\vdots\\
&&&&&&k+1\\
1&\cdots&1&2&\cdots&2&k+2\\
2\mbox{ or } 3\\
\vdots\\
p\\
k+3\\
\vdots
\end{matrix}$$
where $1\leq p\leq k+2$ and the number of $2$'s in the row of length $n$ ranges from $0$ to $m-1$. Now  let $\mathcal{U}$ be the set of all tableaux contributing towards the positive coefficient of some Schur function in the Schur function expansion of $r(n-1,1^k,m+1,1^l)$. Then $\mathcal{U}$ is the set of SSYT $U$ with $\sh(U)=(n-1,1^k,m+1,1^l)$ and $w(U)$ is lattice. Let $\mathcal{U}_1$
consist of those elements of $\mathcal{U}$ of
the form
$$\begin{matrix}
&&&&&&1&\cdots&1\\
&&&&&&2\\
&&&&&&\vdots\\
&&&&&&k+1\\
1&\cdots&1&2&\cdots&2&k+2\\
2\mbox{ or } 3\\
\vdots\\
p\\
k+3\\
\vdots
\end{matrix}$$
where $1\leq p\leq k+2$ and the number of $2$'s 
in the row of length $m+1$ ranges from $0$ to $m-1$. In particular, there is at least one 1 in the
row of length $m+1$.  Let $\mathcal{U}_2$  consist of those elements of $\mathcal{U}$ of
the form
$$\begin{matrix}
&&&1&\cdots&1\\
&&&2\\
&&&\vdots\\
&&&k+1\\
2&\cdots&2&k+2\\
3\\
\vdots\\
p\\
k+3\\
\vdots
\end{matrix}$$
where $2\leq p\leq k+2$. Since $n-1 > m$, these are LR-fillings.  We see that  $\mathcal{U}$ is the disjoint union of $\mathcal{U}_1$ and $\mathcal{U}_2$.
Observe there exists a natural bijection $\phi:\mathcal{T}\rightarrow \mathcal{U}_1$ given by $\phi(T)=U$ if and only if $c(T)=c(U)$ for $T\in \mathcal{T}$ and $U\in \mathcal{U}$. 
Intuitively, $U$ is obtained from $T$ by moving $n-m-1$ copies of $1$ from the row of length $n$ of $T$ to its top row.
From this and \eqref{eq:lrrule} it follows that
\begin{eqnarray*}r(n-1, 1^k, m+1, 1^l)-r(m, 1^k, n, 1^l)&=&\sum _{U\in \mathcal{U}_2} s_{c(U)}\\
&=&\sum _{i=0} ^{\min \{k,l\}}s_{(n-1,m+1,2^i,1^{k+l-2i})}.\end{eqnarray*}


We now prove the second part similarly. Let $\mathcal{T}$ be the set of all tableaux contributing towards the positive coefficient of some Schur function in the Schur function expansion of $r(n,1^k,m,1^l)$. 
Then $\mathcal{T}$ is the set of SSYT $T$  with $\sh(T)=(n, 1^k, m, 1^l)$ and $w(T)$ is lattice.
We can partition $\mathcal{T}$ into two disjoint sets $\mathcal{T}_1$ and $\mathcal{T}_2$ as follows.

Let $\mathcal{T}_1$ consist of those elements of $\mathcal{T}$ of the form
$$\begin{matrix}
&&&&&&1&\cdots&1\\
&&&&&&2\\
&&&&&&\vdots\\
&&&&&&k+1\\
1&\cdots&1&2&\cdots&2&k+2\\
2 \mbox{ or } 3\\
\vdots\\
p\\
k+3\\
\vdots
\end{matrix}$$where $1\leq p\leq k+2$ and the number of $2$'s
in the row of length $m$ ranges from $0$ to $m-2$.
Then $\mathcal{T}_2$ must consist of those elements of $\mathcal{T}$ of the form
$$\begin{matrix}
&&&1&\cdots&1\\
&&&2\\
&&&\vdots\\
&&&k+1\\
2&\cdots&2&k+2\\
3 \\
\vdots\\
p\\
k+3\\
\vdots
\end{matrix}$$
where $2\leq p\leq k+2$. 
Now let $\mathcal{U}$ be the set of all tableaux contributing towards the positive coefficient of some Schur function in the Schur function expansion of $r(m,1^k,n,1^l)$. Then $\mathcal{U}$ is the set of SSYT with $\sh(U)=(m,1^k,n,1^l)$ and $w(U)$ is lattice. We can partition $\mathcal{U}$ into 
three disjoint sets $\mathcal{U}_1$, $\mathcal{U}_2$ and $\mathcal{U}_3$ as follows.  Let $\mathcal{U}_1$ consist of those elements of
$\mathcal{U}$ the form
$$\begin{matrix}
&&&&&&1&\cdots&1\\
&&&&&&2\\
&&&&&&\vdots\\
&&&&&&k+1\\
1&\cdots&1&2&\cdots&2&k+2\\
2 \mbox{ or } 3\\
\vdots\\
p\\
k+3\\
\vdots
\end{matrix}$$where $1\leq p\leq k+2$ and the number of $2$'s 
in the row of length $n$ ranges from $0$ to $m-2$.
Let $\mathcal{U}_2$ consist of those elements of $\mathcal{U}$ of the form
$$\begin{matrix}
&&&&&&1&\cdots&1\\
&&&&&&2\\
&&&&&&\vdots\\
&&&&&&k+1\\
1&\cdots&1&2&\cdots&2&k+2\\
3 \\
\vdots\\
p\\
k+3\\
\vdots
\end{matrix}$$where $1\leq p\leq k+2$ with $p\neq2$, and the number of $2$'s in the row of length $n$ is $m-1$.  Then $\mathcal{U}_3$ must consist of those elements of $\mathcal{U}$  of the form
$$\begin{matrix}
&&&&&&1&\cdots&1\\
&&&&&&2\\
&&&&&&\vdots\\
&&&&&&k+1\\
1&\cdots&1&2&\cdots&2&k+2\\
2 \\
\vdots\\
p\\
k+3\\
\vdots
\end{matrix}$$where $2\leq p\leq k+2$ and  the number of $2$'s
in the row of length $n$ is $m-1$. 
Let $\phi$ be the map that moves $n-m$ copies of 1 from the top row of an element of $\mathcal{T}$
to the row of length $m$.  Observe that $\phi$ is a bijection from $\mathcal{T}_1$ to 
$\mathcal{U}_1$ and from $\mathcal{T}_2$ to $\mathcal{U}_2$.  
From this and \eqref{eq:lrrule} it follows that
\begin{eqnarray*}r(m, 1^k, n, 1^l)-r(n, 1^k, m, 1^l)&=&\sum _{U\in \mathcal{U}_3} s_{c(U)}\\
&=&\sum _{i=0} ^{\min \{k,l-1\}}s_{(n,m+1,2^i,1^{k+l-2i-1})}.\end{eqnarray*}

The third and fourth parts follow by applying the map $\omega$ to the first and second parts, respectively.
\end{proof}

It will turn out that Lemma~\ref{lem:fourcovers} explains all the edges of $\PNl$, and hence 
all the order relations.  We now give a partner lemma that will ultimately show that there are no other order relations in $\PNl$.  Again, the reader may wish to compare these relations with those in 
Figure~\ref{fig:p12_6}.

\begin{lemma}\label{lem:onlycovers} Suppose we have non-negative integers $k,l,k',l'\geq 0$ and $m,n,m',n'\geq 1$ with
the properties that $k+l=k'+l'$ and $m+n=m'+n'$.
\begin{enumerate}
\item If $n-1>m$ then 
$$(n-1, 1^k, m+1, 1^l) \not\leqs (m, 1^{k'}, n, 1^{l'}).$$
\item If $n>m$ and $l \geq 1$ then 
$$(m, 1^k, n, 1^l) \not\leqs (n, 1^{k'}, m, 1^{l'}).$$
\item If $n, n' \geq 2$ and $l>k$ then 
$$(m, 1^k, n, 1^l) \not\leqs (m', 1^l, n', 1^k).$$
\item If $m,n, n' \geq 2$ and $l-1>k$ then 
$$(m, 1^{l-1}, n, 1^{k+1}) \not\leqs (m', 1^k, n', 1^l).$$
\end{enumerate}
\end{lemma}

\begin{proof}
We first prove (1).  We have that 
\begin{eqnarray*}
\rows{(n-1, 1^k, m+1, 1^l)} & = & (n-1, m+1, 1^{k+l}) \\
& \ledom & (n, m, 1^{k'+l'})  \\
& = & \rows{(m, 1^{k'}, n, 1^{l'})}.
\end{eqnarray*}
The result now follows from Lemma~\ref{lem:rowscols}.

Applying the map $\omega$ to (1) gives (3).

We next prove (2).  Consider the partition $\nu = (n, m+1, 1^{k+l-1})$.  One can check that
the ribbon $(m, 1^k, n, 1^l)$ has an LR-filling of content $\nu$.  Indeed, the filling
$$\begin{matrix}
&&&&&&1&\cdots&1\\
&&&&&&2\\
&&&&&&\vdots\\
&&&&&&k+1\\
1&\cdots&1&2&\cdots&2&k+2\\
2 \\
k+3\\
k+4\\
\vdots
\end{matrix}$$ where there are $m-1$ copies of $2$ in the row of length $n$, has the required property.
On the other hand, the ribbon $(n, 1^{k'}, m, 1^{l'})$ has no LR-filling $T$ of content $\nu$.  Indeed, so
that $w(T)$ is lattice, the top row of $T$ contains only $1$'s, leaving $m$ columns to be filled.
$T$ must contain $m+1$ copies of $2$, which
is impossible since we can only put at most one copy of $2$ in each column.
We conclude that $s_\nu$ has positive coefficient in the Schur expansion of $r(m, 1^k, n, 1^l)$ but coefficient 0 in the Schur expansion of $r(n, 1^{k'}, m, 1^{l'})$, yielding the result.

Applying the map $\omega$ to (2) gives (4).
\end{proof}

\section{Poset of rectangles}\label{sec:main} 

We are now in a position to completely characterize when the difference of two multiplicity-free ribbon Schur functions is Schur positive. However, before we do this we will introduce a new notation for ribbon Schur functions that will support the clarity of  the statement of our theorem more than our current notation, which supported the clarity of our proofs.  \new{In effect, our new notation will help explain why $\PNl$ resembles a product of two chains.}

Observe that if $\alpha =(m, 1^k, n, 1^l)$, $m,n\geq 1$, $k,l\geq 0$ and 
$\alpha = \lambda/\mu$, then the natural choice for $\mu$ is $(n-1)^{k+1}$.  Note that we safely ignore the trivial case when $m=n=1$.  Therefore, we can index multiplicity-free ribbons according to the dimensions of the rectangle $\mu$.  More precisely, for a fixed $N$ and $\ell$ we denote the multiplicity-free ribbon Schur functions appearing in the poset $\PNl$ by
$$r[a,b]:=r(N-\ell-b+1,1^{a-1}, b+1, 1^{\ell-a-1})$$for $1\leq a\leq \ell-1$ and $1\leq b\leq N-\ell$, as in
Figure~\ref{fig:indexing}.  
\begin{figure}[htbp]
\begin{center}
\[
\scalebox{.7}{\includegraphics{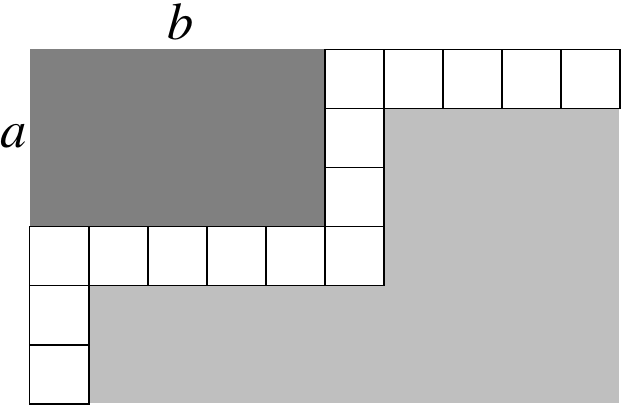}}
\]
\caption{The ribbon in $\multfreeP{15}{6}$ denoted $[3,5]$.}
\label{fig:indexing}
\end{center}
\end{figure}
Some equivalence classes of multiplicity-free ribbons can be indexed by a rectangle in more than one way, and we will need to set a convention.  In particular, if
$a=\ell-1$ then by Proposition~\ref{prop:ribreverse} we have
\begin{eqnarray*}
r[\ell-1, b]&=&r(N-\ell-b+1,1^{\ell-2},b+1)\\
&=&r(b+1,1^{\ell-2},N-\ell-b+1)\\
&=&r[\ell-1, N-\ell-b]
\end{eqnarray*}
and we will choose to use the notation $r[\ell-1,\min\{b,N-\ell-b\}]$.
{Also, if} $b=N-\ell$ then by Proposition~\ref{prop:ribreverse} we have
\begin{eqnarray*}
r[a, N-\ell]&=&r(1^a, N-\ell+1,1^{\ell-a-1})\\
&=&r(1^{\ell-a-1}, N-\ell+1,1^a)\\
&=&r[\ell-a-1, N-\ell]
\end{eqnarray*}
and we will choose to use the notation $r[\min\{a,\ell-a-1\},N-\ell]$.
Note that because our labelling convention requires $a, b \geq 1$,
we also have
$$r(N-\ell+1,1^{\ell-1})=r(1^{\ell-1}, N-\ell+1)=r[\ell-1, N-\ell].$$
In each case, we will let $[a,b]$ denote the equivalence class of ribbons with ribbon Schur function $r[a,b]$.

In summary, we have the following result.

\begin{proposition}\label{prop:rectanglelabel}
The elements of $\PNl$ are those $[a,b]$ such that 
\begin{itemize}
\item $1\leq a < \ell -1$ and $1\leq b< N-\ell$, or
\item $a =\ell -1$ and $1\leq b \leq  \left\lfloor\frac{N-\ell}{2}\right\rfloor$, or
\item $1 \leq a \leq \left\lfloor\frac{\ell-1}{2}\right\rfloor$ and $b = N-\ell$, or
\item $a=\ell-1$ and $b=N-\ell$.
\end{itemize}
\end{proposition}

Figure~\ref{fig:rects} shows the poset $\multfreeP{12}{6}$ from Figure~\ref{fig:p12_6} now with
the elements labelled by their corresponding rectangles.  

\begin{figure}[htbp]
\begin{center}
\[
\scalebox{.5}{\includegraphics{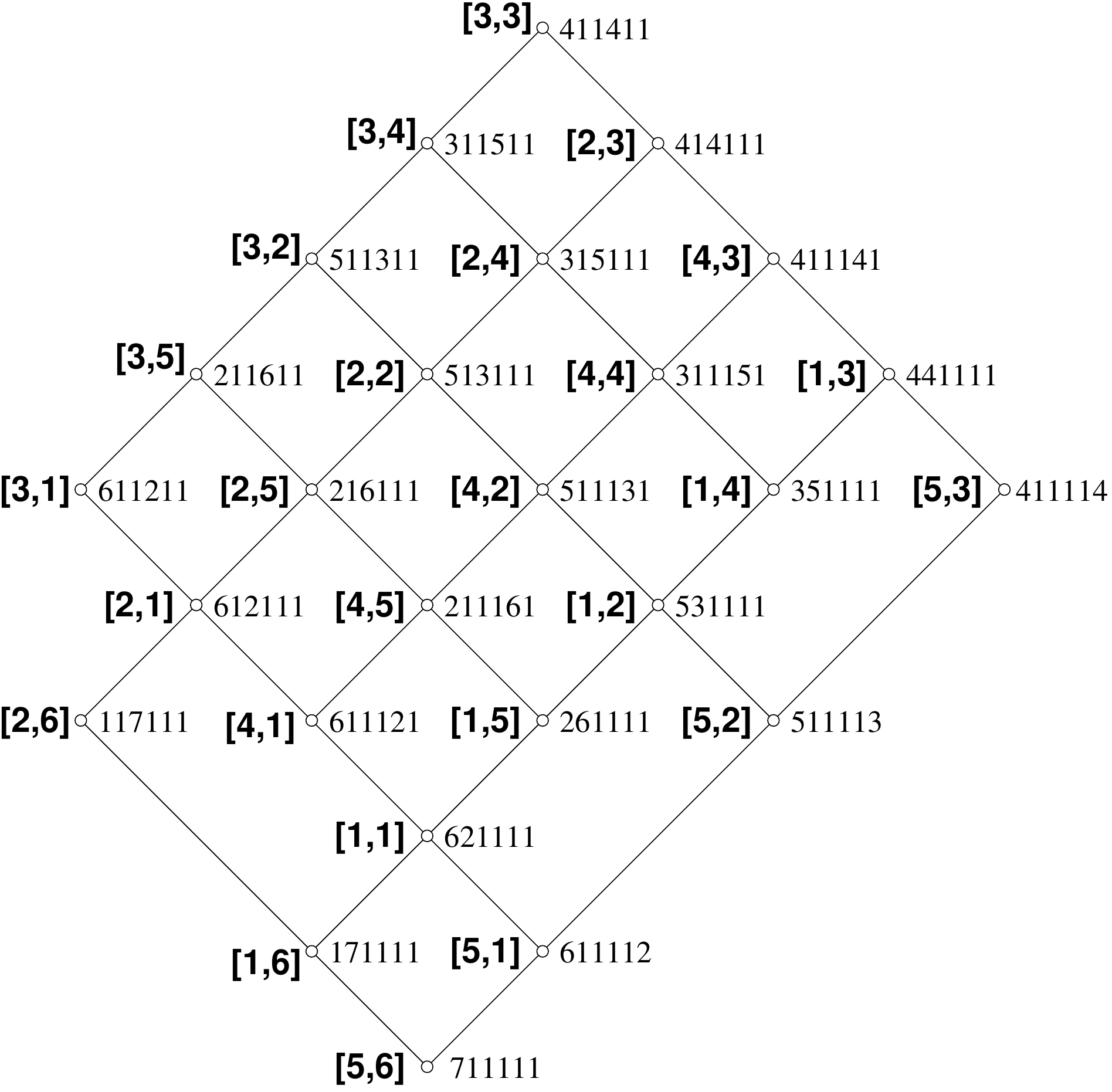}}
\]
\caption{$\multfreeP{12}{6}$ with labels of the form $[a,b]$.}
\label{fig:rects}
\end{center}
\end{figure}
We now define two total orders $<_h$ and $<_w$. Let $<_h$ be the total order on $1, \ldots, \ell-1$ such that
$$\ell-1<_h 1<_h \ell-2 <_h 2 <_h \cdots <_h \left\lfloor \frac{\ell}{2} \right\rfloor.$$
In other words, pick $\varepsilon$ with $0 < \varepsilon < \frac{1}{2}$.  
Then $a <_h b$ if and only if $a$ is further than $b$ from $\frac{\ell}{2} - \varepsilon$ in absolute value.
If $1\leq a,b\leq \ell-1$ then we denote the meet and join of $a$ and $b$ with respect to $<_h$ by $a\wedge _h b$ and $a\vee _h b$, respectively. Let $<_w$ be the total order on $1, \ldots, N-\ell$ such that
$$N-\ell<_w 1<_w N-\ell-1 <_w 2 <_w \cdots  <_w\left\lfloor\frac{N-\ell+1}{2} \right\rfloor.$$
In other words,  $a <_w b$ if and only if $a$ is further than $b$ from $\frac{N-\ell+1}{2} - \varepsilon$ in absolute value.
If $1\leq a,b\leq N-\ell$ then we denote the meet and join of $a$ and $b$ with respect to $<_w$ by $a\wedge _w b$ and $a\vee _w b$, respectively.

Now we are ready for our main theorem, which gives a complete description of the Schur positivity order in $\PNl$.

\begin{theorem}\label{th:bigdiff}
Consider the poset $\PNl$. Then
$$ [a_1, b_1]\leqs [a_2, b_2] \mbox{ if and only if } a_1\leq _h a_2 \mbox{ and }  b_1 \leq _w b_2.$$
\end{theorem}

\begin{proof}
If one interprets this theorem as well as Lemmas~\ref{lem:fourcovers} and \ref{lem:onlycovers}
in the context of a particular $\PNl$ such as $\multfreeP{12}{6}$ in Figure~\ref{fig:rects}, it 
becomes apparent that Theorem~\ref{th:bigdiff} follows directly from the two lemmas.
Even so, it is worthwhile to detail the connection from the lemmas to this theorem.

Parts (1) and (2) of Lemma~\ref{lem:fourcovers} imply that 
$[a, b_1] \les [a, b_2]$ if $b_1 \lessdot_w b_2$. 
Similarly, (3) and (4) of Lemma~\ref{lem:fourcovers} imply that 
$[a_1, b] \les [a_2, b]$ if $a_1 \lessdot_h a_2$.  
Therefore, $[a_1, b_1]\leqs [a_2, b_2]$ if $a_1\leq _h a_2$ and  $b_1 \leq _w b_2$.

To prove the converse, suppose that while $[a_1, b_1]\leqs [a_2, b_2]$, it is not the case that both $a_1\leq _h a_2$ and  $b_1 \leq _w b_2$.  Suppose first that $a_1 \not\leq_h a_2$.  Since
$\leq_h$ is a total order, we must have $a_1 >_h a_2$.  
Define $\widehat{a_2}$ by
$\widehat{a_2} \lessdot_h a_1$.  Then $a_2 \leq_h \widehat{a_2}$ and
so
\begin{equation}\label{eq:chain}
[a_1, b_1] \leqs [a_2, b_2] \leq_s [\widehat{a_2}, b_2].
\end{equation}However, (3) and (4) of Lemma~\ref{lem:onlycovers} imply that
$[a, b] \not\leqs [a', b']$ if $a' \lessdot_h a$, regardless of the relationship between
$b$ and $b'$; this contradicts \eqref{eq:chain} .  Similarly, if we assume that $b_1 \not\leq_w b_2$, 
we can use (1) and (2) of Lemma~\ref{lem:onlycovers} to arrive at a contradiction.
\end{proof}

As an example, notice that in Figure~\ref{fig:rects} the chains 
$5 <_h 1 <_h 4 <_h 2 <_h 3$ and $6 <_w 1 <_w 5 <_w 2 <_w 4 <_w 3$ determine the order relations.

\begin{corollary}
The cover relations in $\PNl$ are given by
$$\begin{array}{rcll}
[a, N-\ell]& \lessdot_s &[a+1, N-\ell] &\mbox{ for }1\leq a\leq \left\lfloor \frac{\ell-1}{2}\right\rfloor-1, \\  
\ [a_1, b] &  \lessdot_s &[a_2, b] &\mbox{ for } a_1 \lessdot_h a_2 \mbox{ and } b < N-\ell, 
\end{array}$$
and
$$\begin{array}{rcll}
[\ell-1, b]&\lessdot_s &[\ell-1, b+1] &\mbox{ for }1\leq b\leq \left\lfloor \frac{N-\ell}{2}\right\rfloor -1,\\
\ [a, b_1]&\lessdot_s &[a, b_2] &\mbox{ for } b_1 \lessdot_w b_2 \mbox{ and } a < \ell-1, 
\end{array}$$
and
$$
[\ell-1, N-\ell]  \lessdot_s  [1, N-\ell],\, [\ell-1,1].
$$
\end{corollary}

\begin{proof}
For the covering relations
involving $[a,b]$ where $a=\ell-1$ or $b=N-\ell$, 
Proposition~\ref{prop:rectanglelabel} is relevant.  
Otherwise, the cover relations follow directly 
from Theorem~\ref{th:bigdiff}.
\end{proof}

If $[a,b], [c,d]\in \PNl$ then we denote their meet and join  with respect to $<_s$ by $[a,b]\wedge _s [c,d]$ and $[a,b]\vee _s [c,d]$, respectively.

\begin{corollary}\label{cor:meetjoin}
The poset $\PNl$ is a lattice, and for $[a,b], [c,d]\in \PNl$ we have
\[
[a,b]\vee _s [c,d] = [a\vee _h c, b\vee _w d].
\]
\[
[a,b]\wedge _s [c,d] = \left\{
\begin{array}{ll}
[a\wedge_h c, N-\ell-(b \wedge_w d)] & \mbox{if $a\wedge_h c = \ell-1$ and $b \wedge_w d > \frac{N-\ell}{2}$}, \\

[\ell-1-(a\wedge_h c), b\wedge_w d ] & \mbox{if $b \wedge_w d = N-\ell$ and $a \wedge_h c > \frac{\ell-1}{2}$},\\

[a\wedge _h c, b\wedge _w d]  & \mbox{otherwise}.
\end{array}
\right.
\]
\end{corollary}

\section{Concluding remarks}\label{sec:conclusion}

\begin{remark}\label{rem:schubert}
If an element $\sigma\in H^\ast(Gr(\ell,\mathbb{C}^{N+1}), \mathbb{Z})$, the cohomology ring of the Grassmannian of $\ell$-dimensional subspaces in $\mathbb{C}^{N+1}$, can be written as a non-negative linear combination of Schubert classes then we say $\sigma$ is Schubert positive. By the discussion in, say,  \cite[Section 4]{Gutschwager} or \cite[Section 4]{KWvW} it follows that the difference of products
\begin{equation}\label{eq:schubert}
\sigma _{(a^b)}\cdot \sigma_{((N-\ell)^{\ell-a-1},(N-\ell-b)^a)}-
\sigma _{(c^d)}\cdot \sigma_{((N-\ell)^{\ell-c-1},(N-\ell-d)^c)}
\end{equation}
is Schubert positive if and only if
$$r[a,b]-r[c,d]$$is Schur positive for $r[a,b],r[c,d]\in \PNl$. Consequently, whether the difference in \eqref{eq:schubert} is Schubert positive or not is completely determined by Theorem~\ref{th:bigdiff}.
The reader may wish to compare the Schubert classes appearing in the first
term in \eqref{eq:schubert} with Figure~\ref{fig:indexing}: $(a^b)$ is clearly the shape 
of the shaded rectangle, while $((N-\ell)^{\ell-a-1},(N-\ell-b)^a)$ is the shape of the other shaded
region rotated 180 degrees.
\end{remark}

\begin{remark}\label{rem:trim}
It is natural to ask what lattice-theoretic properties the poset $\PNl$ possesses.  
As already observed, $\PNl$ has well-defined meet and join operations and so is a lattice.
On the other hand, for example 
from Figure~\ref{fig:rects}, it is clear that $\PNl$ is not graded.
This is caused by, for example, the ribbon $111171$ in $\multfreeP{12}{6}$ 
being equivalent to the ribbon 
$171111$.  If this type of equivalence did not occur in $\PNl$, then the poset that would result would 
be exactly a product of two chains.   It is for this reason that we state in the introduction that 
$\PNl$ is only a slight modification of a product of two chains.

Since $\PNl$ is not graded, it is certainly not distributive.  However, \emph{trim} lattices are introduced
in \cite{Tho} as an ungraded analogue
of distributive lattices, and are a stronger version of \emph{extremal}
lattices defined in \cite{Mar}.  A lattice is said to be trim if it has a maximal chain of
$m+1$ left modular elements, exactly $m$ join-irreducibles, and exactly $m$ meet-irreducibles.
One can show that $\PNl$ is trim with $m=N-3$, and that any element on a chain of maximum length is left modular.

The \emph{spine} of a trim lattice $L$ consists of those elements of $L$ which lie on some chain of
$L$ of maximum length.  It is shown in \cite{Tho} that the spine of a trim lattice $L$ is a 
distributive sublattice of $L$, as is clearly seen to be the case for $\multfreeP{12}{6}$ in Figure~\ref{fig:rects}.
\end{remark}

Remark~\ref{rem:trim}  serves as a fitting conclusion: despite the fact that the Schur positivity order
$\bigP{N}$ seems unstructured, the poset of multiplicity-free ribbons $\PNl$ has an
appealing and intelligible form.


\begin{thebibliography}{130}
\bibliographystyle{plain}

\bibitem{BBR} Fran{\c{c}}ois Bergeron, Riccardo Biagioli, and Mercedes~H. Rosas,
Inequalities between {L}ittlewood-{R}ichardson coefficients,
{\em J. Combin. Theory Ser. A}, {\bf 113} (4) (2006), 567--590.

\bibitem{HDL}
Louis J. Billera, Hugh Thomas, and Stephanie van Willigenburg,
Decomposable compositions, symmetric quasisymmetric functions and 
equality of ribbon Schur functions, {\it Adv.  Math.} {\bf204} (2006), 204--240.

\bibitem{BucSoftware}
Anders~S. Buch,
\newblock {L}ittlewood-{R}ichardson calculator, 1999.
\newblock \newline Available from
  \href{http://www.math.rutgers.edu/~asbuch/lrcalc/}
  {\texttt{http://www.math.rutgers.edu/\symbol{126}asbuch/lrcalc/}}.

\bibitem{FFLP}
Sergey Fomin, William Fulton, Chi-Kwong Li, and Yiu-Tung Poon,
Eigenvalues, singular values, and {L}ittlewood-{R}ichardson coefficients,
{\em Amer. J. Math.}, {\bf 127} (1) (2005), 101--127.

\bibitem{Ful1} William Fulton, \emph{Young Tableaux}, Cambridge University Press, Cambridge, UK, 1997.

\bibitem{Ful2} William Fulton, Eigenvalues, invariant factors, highest weights, and Schubert calculus, 
{\it Bull. Amer. Math. Soc.} {\bf 37} (2000), 209--249.

\bibitem{Gutschwager}
Christian Gutschwager,
On multiplicity-free skew characters and the Schubert calculus,
{\em Ann. Comb.}, to appear. \href{http://www.arxiv.org/abs/math/0608145v2}{\tt arXiv:math.CO/0608145v2}.

\bibitem{How} Roger Howe, Perspectives on invariant theory: Schur duality, multiplicity-free actions and beyond.  (English summary)   The Schur lectures (1992) (Tel Aviv), 
{\em Israel Math. Conf. Proc.}, {\bf 8} (1995), 1--182.


\bibitem{KWvW}
Ronald C. King, Trevor A. Welsh, and  Stephanie J. van Willigenburg, Schur positivity of skew Schur function differences and applications to ribbons and Schubert classes, 
\href{http://www.arxiv.org/abs/0706.3253v2}{\tt arXiv:0706.3253v2}.

\bibitem{Kir}
Anatol~N. Kirillov, An invitation to the generalized saturation conjecture,
{\em Publ. Res. Inst. Math. Sci.}, {\bf 40} (4) (2004), 1147--1239.

\bibitem{LPP}
Thomas Lam, Alexander Postnikov, and Pavlo Pylyavskyy, 
{S}chur positivity and {S}chur log-concavity, {\em Amer. J. Math.},
to appear. \href{http://www.arxiv.org/abs/math.CO/0502446v3}
  {\texttt{arXiv:math.CO/0502446v3}}.

\bibitem{LLT}
Alain Lascoux, Bernard Leclerc, and Jean-Yves Thibon, Ribbon tableaux, {H}all-{L}ittlewood functions, quantum affine algebras, and unipotent varieties,
{\em J. Math. Phys.}, {\bf 38} (2) (1997), 1041--1068.

\bibitem{Mac} 
I.~G. Macdonald, {\em Symmetric functions and {H}all polynomials}, 
Oxford Mathematical Monographs, The Clarendon Press Oxford University
Press, New York, second edition, 1995.

\bibitem{Mar}
George Markowsky, Primes, irreducibles and extremal lattices, {\em Order} {\bf 9} (1992), 265--290.

\bibitem{McN}
Peter R. W. McNamara,  Necessary conditions for Schur-positivity, \href{http://www.arxiv.org/abs/0706.1800v1}{\tt arXiv:0706.1800v1}.

\bibitem{HDL3} 
Peter R. W. McNamara and Stephanie van Willigenburg, Towards a combinatorial classification of skew Schur functions, \href{http://www.arxiv.org/abs/math/0608446}{\tt arXiv:math/0608446v2}.

\bibitem{Oko}
Andrei Okounkov, Log-concavity of multiplicities with application to characters of
 {${\rm U}(\infty)$}, {\em Adv. Math.}, {\bf 127} (2) (1997), 258--282.

\bibitem{HDL2}
Victor Reiner, Kristin M. Shaw, and Stephanie van Willigenburg,  Coincidences among skew Schur functions, {\it Adv.  Math.}, {\bf 216} (1) (2007), 118--152.

\bibitem{ECII}
Richard P. Stanley,
\emph{Enumerative Combinatorics, Volume~2},
Cambridge University Press, Cambridge, UK, 1999. 

\bibitem{Ste}
John~R. Stembridge, 
Multiplicity-free products of Schur functions, {\em Ann. Comb.},  {\bf 5} (2) (2001), 113--121. 

\bibitem{SteSoftware}
John~R. Stembridge,
\newblock {SF}, posets and coxeter/weyl.
\newblock \newline Available from
  \href{http://www.math.lsa.umich.edu/~jrs/maple.html}
  {\texttt{http://www.math.lsa.umich.edu/\symbol{126}jrs/maple.html}}.

\bibitem{Tho}
Hugh Thomas,
An analogue of distributivity for ungraded lattices,
{\it Order} {\bf 23} (2-3) (2006), 249--269.

\bibitem{ThomasYong}
Hugh Thomas and Alexander Yong,
Multiplicity-free Schubert calculus, {\it Canad. Math. Bull.}, to appear.
\href{http://www.arxiv.org/abs/math/0511537v2}{\tt arXiv:math.CO/0511537v2}.
\end{thebibliography}
\end{document}